\documentclass[oneside,reqno]{amsart}

\usepackage[utf8]{inputenc} 

\usepackage[top=1.0 in, left=1.0 in, bottom=1.0 in, right=1.0 in]{geometry}
\usepackage{graphicx}

\usepackage{mathrsfs}
\usepackage{manfnt}

\usepackage{lipsum}

\usepackage{standalone}
\usepackage{etoolbox}
\usepackage{varwidth}
\usepackage{emptypage}

\usepackage{amsmath}
\usepackage{accents}
\usepackage{amsthm}
\usepackage{bbm}
\usepackage{thmtools}
\usepackage{amssymb}
\usepackage{latexsym}
\usepackage{geometry}
\usepackage{setspace}
\usepackage{manfnt}
\usepackage{mathdots}
\usepackage{cancel}
\usepackage{indentfirst}

\usepackage{algpseudocode}
\usepackage{algorithm}

\usepackage{ragged2e}
\usepackage{enumerate}
\usepackage{listings}
\usepackage{tabularx}
\usepackage{xparse}
\usepackage{xspace}
\usepackage[shortlabels]{enumitem}
\usepackage{multirow}
\usepackage{todonotes}
\usepackage{dcolumn}
\usepackage{array}
\usepackage{colortbl, xcolor}
\usepackage{longtable}
\usepackage{tabu}
\usepackage{python}

\usepackage{tikz}
\usetikzlibrary{positioning,chains,fit,shapes,calc,arrows,patterns,cd,knots,hobby} 

\usepackage{asymptote}

\usepackage{ytableau}

\usepackage{blkarray}

\allowdisplaybreaks

\usepackage{graphicx}
\setkeys{Gin}{width=\linewidth,totalheight=\textheight,keepaspectratio}
\graphicspath{{Graphics/}}

\usepackage[hidelinks]{hyperref}
\hypersetup{colorlinks=false}

\usepackage{fancyvrb}
\fvset{fontsize=\normalsize}

\usepackage{xspace}

\usepackage{units}

\declaretheorem[numberwithin=section,name=Theorem]{theorem}
\declaretheorem[name=Lemma,sibling=theorem]{lemma}
\declaretheorem[name=Corollary,sibling=theorem]{corollary}
\declaretheorem[name=Definition,sibling=theorem,style=definition]{definition}
\declaretheorem[name=Proposition,sibling=theorem]{proposition}

\declaretheorem[name=Example,sibling=theorem,style=definition]{example}

\declaretheorem[name=Note,style=remark]{note}
\declaretheorem[name=Remark,sibling=note,style=remark]{remark}

\newenvironment{notation}[1][Notation]{\begin{trivlist}
\item[\hskip \labelsep {\bfseries #1}.]}
{\end{trivlist}}

\setlist[enumerate]{label=(\arabic*)}

\newlist{exercises}{enumerate}{1}
\setlist[exercises]{label=Ex. \arabic*}

\def \C {\mathbb{C}}
\newcommand{\Z}{\mathbb{Z}}

\newcommand{\thus}{{.\raise 4pt\hbox{.}.\;}}

 2

\newcommand\mapsfrom{\mathrel{\reflectbox{\ensuremath{\mapsto}}}}

\DeclareMathOperator{\supp}{supp}

\DeclareMathOperator{\gl}{\mathfrak{gl}}

\DeclareMathOperator{\Aut}{Aut}

\DeclareMathOperator{\Frac}{Frac}

\DeclareMathOperator{\End}{End}
\DeclareMathOperator{\K}{\mathbbm{k}}

\usepackage[citestyle=alphabetic,bibstyle=alphabetic,sorting=nyt]{biblatex}
\begin{document}

\title{Maps between standard and principal flag orders}
\author{Erich C. Jauch}
\date{}

\address{Department of Mathematics \& Physics, Westminster College (Missouri), Fulton, MO-65251, USA}
\email{erich.jauch@westminster-mo.edu}
\urladdr{ecjauch.com}

\maketitle


\begin{abstract}
Galois orders, introduced in 2010 by V. Futorny and S. Ovsienko,
form a class of associative algebras that contain many important examples,
such as the enveloping algebra of $\gl_n$ (as well as its quantum deformation),
generalized Weyl algebras, and shifted Yangians. The main motivation for introducing Galois
orders is they provide a setting for studying certain infinite dimensional irreducible
representations, called Gelfand-Tsetlin modules. Principal Galois orders, defined by J.
Hartwig in 2017, are Galois orders with an extra property, which makes them easier to study.
All of the mentioned examples are principal Galois orders. In 2019, B. Webster defined
principal flag orders which in most situations are Morita equivalent to principal Galois orders, and further simplifies their study. This paper describes some techniques to connect the study pairs
of standard flag and Galois orders with related data. Such techniques include: a sufficient
condition for morphisms between standard flag and Galois orders to exist, a property of flag orders related to differential
operators on affine varieties, and constructing tensor products of standard and principal flag and Galois orders.
\end{abstract}



\section{Introduction}\label{sec: introduction}

In 2010, Futorny and Ovsienko introduced the notion of a Galois order \cite{FO10}, a class of objects
consisting of pairs $(\mathscr{U},\Gamma)$ of an integral domain $\Gamma\subset\mathscr{U}$ an associative
(noncommutative) $\C$-algebra. These pairs generalizes the relation of $U(\gl_n)$ and its \emph{Gelfand-Tsetlin}
subalgebra $\C\left\langle\bigcup_{k=1}^nZ(U(\gl_k))\right\rangle$. Many important objects have been shown to be
members of this collection including: generalized Weyl algebras \cite{BavulaGWA},\cite{rosenberg_1995}, $U(\gl_n)$,
shifted Yangians and finite $W$-algebras \cite{FMO10}, $U_q(\gl_n)$ \cite{FH14}, Coloumb branches \cite{Webster19},
the spherical subalgebra of rational Cherednik algebras of imprimitive complex reflection groups $G(\ell,p,n)$ \cite{lepage2019rational}.
The primary motivation for creating these objects is to unify the study of their Gelfand-Tsetlin modules \cite{early_mazorchuk_vishnyakova_2018}, \cite{futorny_grantcharov_ramirez_2015},\cite{FUTORNY20183182},\cite{SilverthorneWebster20}.

The current research in this area has been focused on principal Galois orders (Definition \ref{def: principal and co-principal Galois orders}), as they
contain all of the examples of interest and the condition that
$X(\Gamma)\subseteq\Gamma$ is much easier to verify. In particular, in 2019, Webster
showed that most principal Galois orders are Morita equivalent to a \emph{principal flag
order} (Definition \ref{def: principal flag order}) which is a Galois order in which the $G$ is trivial and $\mathscr{M}$
is the semidirect product of the group and monoid from the original data (see Lemma 2.5
in \cite{Webster19}). In particular, the data is almost the same, except that $\Lambda$
is assumed to be Noetherian. These flag orders prove easier to study once more as they are
no longer subalgebras of group invariant algebras. One particular class of principal flag order is the standard flag order (Definition \ref{def: standard flag order}) which is the largest principal flag order with a given set of data as it contains every principal flag order with a given set of data.

One of the first directions that one takes after defining a new object is to describe morphisms between these objects. This is particularly important in order to study these objects in the realm of category theory. In this paper, we endeavor to take those first steps into constructing and studying related standard flag orders. We describe a sufficient condition for such maps to exist in Theorem \ref{thm: morphisms sufficient condition}. Also in Section \ref{sec: morphisms}, we prove that certain short exact sequences give rise to embeddings of standard flag orders (Theorem \ref{thm: standard order intersection}). This allows us to prove a property for standard flag orders similar to one of differential operators on polynomial rings (Theorem \ref{thm: standard order quotient}). In Section \ref{sec: tensors}, we construct tensor products of standard flag orders and prove the existence of a chain of embeddings (Theorem \ref{thm: tensor of std orders}). We use this result to show that principal flag orders (and principal Galois orders) are closed under tensor products (Corollaries \ref{cor: principal flag orders closed under tensor} and \ref{cor: principal galois orders closed under tensor} respectively).

\subsection{Galois orders}\label{sec: Galois Orders}
Galois orders were introduced in \cite{FO10}. We will be following
the set up from \cite{HARTWIG2020106806}. Let $\Lambda$ be an integrally closed
domain, $G$ a finite subgroup of $\Aut(\Lambda)$, and $\mathscr{M}$ a
submonoid of $\Aut(\Lambda)$. We will adhere to the following
assumptions for the entire paper:
\begin{align*}
{\rm (A1)}\quad & (\mathscr{M}\mathscr{M}^{-1})\cap G=1_{\Aut{}{\Lambda}} & \text{(\emph{separation})}\\
{\rm (A2)}\quad & \forall g\in G, \forall\mu\in\mathscr{M}\colon {}^g\mu=g\circ\mu\circ g^{-1}\in\mathscr{M} &
\text{(\emph{invariance})}\\
{\rm (A3)}\quad & \Lambda \text{ is Noetherian as a module over } \Lambda^G & \text{(\emph{finiteness})}
\end{align*}

Let $L=\Frac(\Lambda)$ and $\mathscr{L}=L\#\mathscr{M}$, the skew monoid ring, which is defined as the
free left $L$-module on $\mathscr{M}$ with multiplication given by
$a_1\mu_1\cdot a_2\mu_2=(a_1\mu_1(a_2))(\mu_1\mu_2)$ for $a_i\in L$ and $\mu_i\in\mathscr{M}$.
As $G$ acts on $\Lambda$ by automorphisms, we can easily extend this action to $L$, and by {\rm (A2)},
$G$ acts on $\mathscr{L}$. So we consider the following $G$-invariant subrings
$\Gamma=\Lambda^G$, $K=L^G$, and $\mathscr{K}=\mathscr{L}^G$.

A benefit of these assumptions is the following lemma.

\begin{lemma}[\cite{HARTWIG2020106806}, Lemma 2.1 (ii), (iv) \& (v)]\label{lem: Hartwig big lemma}
\item
\begin{enumerate}[\rm (i)]
\item $K=\Frac(\Gamma)$.
\item $\Lambda$ is the integral closure of $\Gamma$ in $L$.
\item $\Lambda$ is a finitely generated $\Gamma$-module and a Noetherian ring.
\end{enumerate}
\end{lemma}

What follows are some definitions and propositions from \cite{FO10}.

\begin{definition}[\cite{FO10}]
A finitely generated $\Gamma$-subring $\mathscr{U}\subseteq\mathscr{K}$ is called a \emph{Galois $\Gamma$-ring}
(or \emph{Galois ring with respect to $\Gamma$}) if $K\mathscr{U}=\mathscr{U}K=\mathscr{K}$.
\end{definition}

\begin{definition}
Let $u\in\mathscr{L}$ such that $u=\sum_{\mu\in\mathscr{M}}a_\mu \mu$. The \emph{support of $u$ over $\mathscr{M}$} is the following:
\[
\supp u=\bigg\{\mu\in\mathscr{M}~\Big\vert~a_\mu\neq0\text{ for }u=\sum_{\mu\in\mathscr{M}}a_\mu \mu\bigg\}
\]
\end{definition}

\begin{proposition}[\cite{FO10}, Proposition 4.1]\label{prop: Gamma Ring Alt Conditions}
Assume a $\Gamma$-ring $\mathscr{U}\subseteq\mathscr{K}$ is generated by $u_1,\ldots,u_k\in \mathscr{U}$.
\begin{enumerate}[\rm (1)]
\item If $\bigcup_{i=1}^k\supp u_i$ generate $\mathscr{M}$ as a monoid, then $\mathscr{U}$ is a Galois ring.
\item If $L\mathscr{U}=L\#\mathscr{M}$, then $\mathscr{U}$ is a Galois ring.
\end{enumerate}
\end{proposition}

\begin{theorem}[\cite{FO10},  Theorem 4.1 (4)]\label{thm: Center of a Galois Ring}
Let $\mathscr{U}$ be a Galois $\Gamma$-ring. Then the center $Z(\mathscr{U})$ of the algebra
$\mathscr{U}$ equals $\mathscr{U}\cap K^{\mathscr{M}}$,
where $K^{\mathscr{M}}=\{k\in K\mid \mu(k)=k~\forall\mu\in\mathscr{M}\}$
\end{theorem}

\begin{definition}[\cite{FO10}]\label{def: Galois Order defintion}
A Galois $\Gamma$-ring $\mathscr{U}$ in $\mathscr{K}$ is a \emph{left} (respectively \emph{right})
\emph{Galois $\Gamma$-order in $\mathscr{K}$} if for any finite-dimensional left (respectively right)
$K$-subspace $W\subseteq\mathscr{K}$, $W\cap\mathscr{U}$ is a finitely generated left (respectively
right) $\Gamma$-module. A Galois $\Gamma$-ring $\mathscr{U}$ in $\mathscr{K}$ is a \emph{Galois
$\Gamma$-order in $\mathscr{K}$} if $\mathscr{U}$ is a left and right Galois $\Gamma$-order in
$\mathscr{K}$.
\end{definition}

\begin{definition}[\cite{DFO94}]\label{def: Harish-Chandra subalg}
Let $\Gamma\subset\mathscr{U}$ be a commutative subalgebra. $\Gamma$ is called a
\emph{Harish-Chandra subalgebra} in $\mathscr{U}$ if for any $u\in\mathscr{U}$, $\Gamma u\Gamma$
is finitely generated as both a left and as a right $\Gamma$-module.
\end{definition}

Let $\mathscr{U}$ be a Galois ring and $e\in\mathscr{M}$ the unit element. We denote
$\mathscr{U}_e=\mathscr{U}\cap Le$.

\begin{theorem}[\cite{FO10}, Theorem 5.2]\label{thm: Galois Order condition}
Assume that $\mathscr{U}$ is a Galois ring, $\Gamma$ is finitely generated and $\mathscr{M}$
is a group.
\begin{enumerate}[\rm (1)]
\item Let $m\in\mathscr{M}$. Assume $m^{-1}(\Gamma)\subseteq\Lambda$ (respectively $m(\Gamma)\subseteq\Lambda$).
Then $\mathscr{U}$ is right (respectively left) Galois order if and only if $\mathscr{U}_e$ is an
integral extension of $\Gamma$.
\item Assume that $\Gamma$ is a Harish-Chandra subalgebra in $\mathscr{U}$. Then $\mathscr{U}$ is
a Galois order if and only if $\mathscr{U}_e$ is an integral extension of $\Gamma$.
\end{enumerate}
\end{theorem}

The following are some useful results from \cite{HARTWIG2020106806}.

\begin{proposition}[\cite{HARTWIG2020106806}, Proposition 2.14]\label{prop: Gamma maxl comm in a Galois Order}
$\Gamma$ is maximal commutative in any left or right Galois $\Gamma$-order $\mathscr{U}$ in $\mathscr{K}$.
\end{proposition}

\begin{lemma}[\cite{HARTWIG2020106806}, Lemma 2.16]\label{lem: Order Containment Implication}
Let $\mathscr{U}_1$ and $\mathscr{U}_2$ be two Galois $\Gamma$-rings in
$\mathscr{K}$ such that $\mathscr{U}_1\subseteq\mathscr{U}_2$. If
$\mathscr{U}_2$ is a Galois $\Gamma$-order, then so too is $\mathscr{U}_1$.
\end{lemma}

It is common to write elements of $L$ on the right side of elements of $\mathscr{M}$.

\begin{definition}
For $X=\sum_{\mu\in\mathscr{M}}\mu\alpha_\mu\in\mathscr{L}$ and $a\in L$
defines the \emph{evaluation of $X$ at $a$} to be
\[
X(a)=\sum_{\mu\in\mathscr{M}}\mu(\alpha_\mu\cdot a)\in L.
\]
Similarly defined is \emph{co-evaluation} by
\[
X^\dagger(a)=\sum_{\mu\in\mathscr{M}}\alpha_\mu\cdot(\mu^{-1}(a))\in L
\]
\end{definition}

The following was independently defined by \cite{Vishnyakova17} called the \emph{universal ring}. 

\begin{definition}
The \emph{standard Galois $\Gamma$-order} is as follows:
\[
\mathscr{K}_\Gamma:=\{X\in\mathscr{K}\mid X(\gamma)\in\Gamma~\forall\gamma\in\Gamma\}.
\]
Similarly we define the \emph{co-standard Galois $\Gamma$-order} by
\[
{}_\Gamma{\mathscr{K}}:=\{X\in\mathscr{K}\mid X^\dagger(\gamma)\in\Gamma
~\forall\gamma\in\Gamma\}.
\]
\end{definition}

\begin{definition}\label{def: principal and co-principal Galois orders}
Let $\mathscr{U}$ be a Galois $\Gamma$-ring in $\mathscr{K}$. If $\mathscr{U}\subseteq\mathscr{K}_\Gamma$ (resp.
$\mathscr{U}\subseteq{}_\Gamma{\mathscr{K}}$), then $\mathscr{U}$ is called a \emph{principal} (resp.
\emph{co-principal}) \emph{Galois $\Gamma$-order}.
\end{definition}

In \cite{HARTWIG2020106806} it was shown that any (co-)principal Galois $\Gamma$-order is a
Galois order in the sense of Definition \ref{def: Galois Order defintion}.


\subsection{Flag orders}

In the notation of flag orders, the group $G$ is denoted by $W$ instead.

\begin{definition}\label{def: principal flag order}
A \emph{principal flag order} with data $(\Lambda,W,\mathscr{M})$ is a
subalgebra of $F\subset\Frac(\Lambda)\#(W\ltimes\mathscr{M})$ such that:
\begin{enumerate}[\rm (i)]
\item $\Lambda\#W\subset F$,
\item $\Frac(\Lambda)F=\Frac(\Lambda)\#(W\ltimes\mathscr{M})$,
\item For every $X\in F$, $X(\Lambda)\subset\Lambda$.\label{item: principal flag order property}
\end{enumerate}
\end{definition}

\begin{definition}\label{def: standard flag order}
The \emph{standard flag order} with data $(\Lambda,W,\mathscr{M})$ is the subalgebra
of all elements $X\in\Frac(\Lambda)\#(W\ltimes\mathscr{M})$ satisfying
(\ref{item: principal flag order property}) and is denoted $\mathcal{F}_\Lambda$.
\end{definition}

\begin{example}
Let $\Lambda=\C[x_1,x_2,\ldots,x_n]$, $W\leq GL(\C^n)$ a complex reflection group (e.g. $W=S_n$), $\mathscr{M}=\Z^n$. Then $\mathcal{F}_\Lambda$ is the degenerate double affine nilHecke algebra associated to $W$ \cite{KumarBook}.
\end{example}

Recall the definition of a standard flag order (see Definition \ref{def: principal flag order}). Let $(\Lambda,W,\mathscr{M})$
be our data, $\mathcal{F}=\Frac(\Lambda)\#(W\ltimes\mathscr{M})$, and $\mathcal{F}_\Lambda$ be the corresponding standard flag order. In this section, we study morphisms between standard flag orders. One motivation for this is future applications to representation theory, via restriction/induction functors.

\begin{notation}
For simplicity, $W\ltimes\mathscr{M}$ is written as $\hat{W}$.
\end{notation}

\begin{example}
If $\Lambda=\C[x_1,x_2,\ldots,x_n]$ and $\hat{W}$ is a finite complex reflection group action on $\C^n$, then $\mathcal{F}_\Lambda$ is the nilHecke algebra of $\hat{W}$ (see \cite{Webster19}).
\end{example}


\section{Morphisms}\label{sec: morphisms}

\subsection{A sufficient condition}

Let $(\Lambda_1,W_1,\mathscr{M}_1)$, $(\Lambda_2,W_2,\mathscr{M}_2)$ be two flag order data, $L_i$
the field of fractions of $\Lambda_i$ for $i=1,2$ and $\mathcal{F}_{\Lambda_i}$ denote the corresponding
standard flag orders. Recall in particular that $\hat{W}_i=W_i\ltimes\mathscr{M}_i$ acts faithfully on $\Lambda_i$.

\begin{theorem}\label{thm: morphisms sufficient condition}
Let $\varphi:\Lambda_1\to\Lambda_2$ be a ring homomorphism and $\psi:\hat{W}_1\to \hat{W}_2$ be a group homomorphism such that
\begin{equation}\label{eq:phipsi-condition}
\varphi\big(w(a)\big)=\psi(w)\big(\varphi(a)\big),\qquad \forall a\in\Lambda_1,\forall w\in \hat{W}_1.
\end{equation}
\begin{enumerate}[{\rm (i)}]
\item\label{item: sufficient condition algebra homomorphism}  There is an algebra homomorphism 
\begin{equation}
\Phi: L_1\# \hat{W}_1\to L_2\# \hat{W}_2
\end{equation}
given by
\begin{equation}
\Phi(fw)=\varphi(f)\psi(w),\qquad f\in L_1, w\in \hat{W}_1
\end{equation}

\item\label{item: sufficient condition map restricts to standard orders} Suppose there is a subspace $U$ of $\Lambda_2$ such that $\Lambda_2\cong \varphi(\Lambda_1)\otimes U$ as $\psi(\hat{W}_1)$-modules, where $\psi(\hat{W}_1)$ acts on $\varphi(a)\otimes u$ by 
\[
\psi(w)\big(\varphi(a)\otimes u\big) = \psi(w)\big(\varphi(a)\big)\otimes u = \varphi(w(a))\otimes u.
\]
Then $\Phi$ restricts to an algebra homomorphism
\begin{equation}
\Phi: \mathcal{F}_{\Lambda_1} \to \mathcal{F}_{\Lambda_2}
\end{equation}
\end{enumerate}
\end{theorem}

\begin{proof}
\ref{item: sufficient condition algebra homomorphism}
$L_i\# \hat{W}_i= L_i\otimes_{\K} \hat{W}_i$ as $(L_i,\hat{W}_i)$-bimodules, so it suffices to show that
$\Phi$ preserves the relation $wf=w(f)w$ for all $w\in \hat{W}_1, f\in L_1$. This relation is preserved iff
$\psi(w)\varphi(a)=\varphi(w(a))\psi(w)$ for all $w\in \hat{W}_1$ and $a\in \Lambda_1$. The left-hand side equals
$\psi(w)\big(\varphi(a)\big) \psi(w)$ so the identity is equivalent to \eqref{eq:phipsi-condition}.

\ref{item: sufficient condition map restricts to standard orders}
Let $X=\sum_{w\in \hat{W}_1}f_w w \in \mathcal{F}_{\Lambda_1}$. By assumption, any element of $\Lambda_2$ is a sum of elements of
the form $b=\varphi(a)\otimes u$, where $a\in\Lambda_1$ and $u\in U$. We have

\[
\Phi(X)(b) = \sum_{w\in \hat{W}_1} \varphi(f_w)\psi(w) \big( \varphi(a)\otimes u\big)
\]
By assumption on how $\psi(W_1)$ acts on such tensors, this equals

\[
\sum_{w\in \hat{W}_1} \varphi(f_w) \big(\varphi(w(a))\otimes u\big)=
\varphi\big(\sum_{w\in \hat{W}_1} f_w w(a)\big)\otimes u \in \varphi(\Lambda_1)\otimes U = \Lambda_2
\]

Thus $\Phi(X)\in\mathcal{F}_{\Lambda_2}$
\end{proof}

\begin{example}\label{ex: Klein 4->S4}
Consider two sets of flag order data. The first is $(\C[x,y],V,\mathbbm{1})$ where
$V=\langle \tau_x,\tau_y\rangle$ is the Klein four-group acting by $\tau_x\colon f(x,y)\mapsto f(-x,y)$ and 
$\tau_y\colon f(x,y)\mapsto f(x,-y)$. The second is $(\C[x_1,x_2,x_3,x_4],S_4,\mathbbm{1})$ $S_4$ is the
symmetric group on 4 elements acting on $\C[x_1,x_2,x_3,x_4]$ by permutation of variables.

Our two homomorphisms are:
\begin{align*}
\varphi&\colon\C[x,y]\rightarrow\C[x_1,x_2,x_3,x_4] & \text{by }& x\mapsto x_2-x_1\text{ and }y\mapsto x_4-x_3,\\
\psi&\colon V\rightarrow S_4 & \text{by } & \tau_x\mapsto (12)
\text{ and }\tau_y\mapsto (34).
\end{align*}
Together they both satisfy equation \ref{eq:phipsi-condition}. Our subspace to show part \ref{item: sufficient condition map restricts to standard orders} is $U=\C[x_1+x_2,x_3+x_4]$. Now by Theorem \ref{thm: morphisms sufficient condition} there is an algebra homomorphism
\[
\Phi\colon\mathcal{F}_{\C[x,y]}=\C[x,y]\langle\frac{1}{x}(\tau_x-\mathbbm{1}),\frac{1}{y}(\tau_y-\mathbbm{1})\rangle\rightarrow
\mathcal{F}_{\C[x_1,x_2,x_3,x_4]}=\C[x_1,x_2,x_3,x_4]\langle\frac{1}{x_{i+1}-x_i}((i,i+1)-\mathbbm{1})\mid i\in\{1,2,3\}\rangle
\]
that sends
\[
\frac{1}{x}(\tau_x-\mathbbm{1})\mapsto\frac{1}{x_2-x_1}((12)-\mathbbm{1})\quad\text{and}\quad\frac{1}{y}(\tau_y-\mathbbm{1})\mapsto\frac{1}{x_4-x_3}((34)-\mathbbm{1}).
\]
\end{example}

\begin{example}
Let $\Lambda=\C[x_1,x_2,\ldots,x_n]$, and $\mathcal{F}_\Lambda$ be the standard flag order corresponding
to $(\Lambda,S_n,\mathbbm{1})$ and $\widetilde{\mathcal{F}_\Lambda}$ be the standard flag order corresponding to
$(\Lambda,A_n,\mathbbm{1})$, where $S_n$ and $A_n$ are the symmetric group and alternating group respectively.
In this situation, $\varphi$ is the identity map, and $\psi\colon A_n\rightarrow S_n$ is the
natural embedding. This gives us  $\Phi\colon\widetilde{\mathcal{F}_\Lambda}\rightarrow\mathcal{F}_\Lambda$. Recall
that $\mathcal{F}_\Lambda$ is the nilHecke algebra of $S_n$ (see Example \ref{ex: standard order is nilhecke algebra}
above and \cite{Webster19}). Thus we will define $\widetilde{\mathcal{F}_\Lambda}$ as the nilHecke algebra of $A_n$.
\end{example}

\begin{example}\label{ex: U(gln) std flag order maps}
Define $\Lambda_n := \C[x_{ji}\mid 1\leq i\leq j\leq n]$, $\mathbb{S}_n=S_1\times S_2\times\cdots\times S_n$ where $S_j$
is the symmetric group on $j$ elements acting by permutation of the variables, $\mathscr{M}_n:=\Z^{n(n-1)/2}=\langle\delta^{ji}\mid 1\leq i\leq j\leq n-1\rangle$ written multiplicatively with the following action:
\[
\delta^{ji}(x_{k\ell})=x_{k\ell}-\delta_{jk}\delta_{i\ell}
\]
All of which come from the Galois order realization of $U(\gl_n)$ from \cite{FO10}. Let $\varphi_n\colon\Lambda_n\rightarrow\Lambda_{n+1}$
and $\psi\colon\mathbb{S}_n\rightarrow\mathbb{S}_{n+1}$ be the standard embeddings and observe that
$\Lambda_{n+1}=\varphi_n(\Lambda_{n})\otimes\C[x_{n+1,i}\mid 1\leq i\leq n+1]$. All of the conditions for
Theorem \ref{thm: morphisms sufficient condition} are met, so we have a map
$\Phi_n\colon\mathcal{F}_{\Lambda_n}\rightarrow\mathcal{F}_{\Lambda_{n+1}}$ where $\mathcal{F}_{\Lambda_n}$ is the standard
flag order with data $(\Lambda_n,\mathbb{S}_n,\mathscr{M}_n)$. Moreover,
\[
\mathcal{F}_{\Lambda_1}\xrightarrow{\Phi_1}\mathcal{F}_{\Lambda_2}\xrightarrow{\Phi_2}\mathcal{F}_{\Lambda_3}
\xrightarrow{\Phi_3}\cdots
\]
By Lemma 2.3 in \cite{Webster19} the standard Galois order $\mathcal{K}_{\Gamma_n}$ is isomorphic to
the spherical subalgebra $e_n\mathcal{F}_{\Lambda_n}e_n$ where
$e_n=\frac{1}{\#\mathbb{S}_n}\sum_{\sigma\in\mathbb{S}_n}\sigma$ is the symmetrizing idempotent. It was shown in \cite{HARTWIG2020106806} that $U(\gl_n)$ is a principle Galois order. Thus we have the following
\[
\begin{tikzpicture}
\node(F1)                           {$\mathcal{F}_{\Lambda_1}$};
\node(F2)   [right=0.75cm of F1]    {$\mathcal{F}_{\Lambda_2}$};
\node(F3)   [right=0.75cm of F2]    {$\mathcal{F}_{\Lambda_3}$};
\node(F4)   [right=0.75cm of F3]    {$\cdots$};
\node(K1)   [below=0.75cm of F1]    {$\mathcal{K}_{\Gamma_1}$};
\node(K2)   [right=0.75cm of K1]    {$\mathcal{K}_{\Gamma_2}$};
\node(K3)   [right=0.75cm of K2]    {$\mathcal{K}_{\Gamma_3}$};
\node(K4)   [right=0.75cm of K3]    {$\cdots$};
\node(U1)   [below=0.75cm of K1]    {$U(\gl_1)$};
\node(U2)   [below=0.75cm of K2]    {$U(\gl_2)$};
\node(U3)   [below=0.75cm of K3]    {$U(\gl_3)$};
\node(U4)   [below=0.95cm of K4]    {$\cdots$};

\foreach \i in {1,...,4} {
\draw[>=stealth,->,black] (K\i) -- (F\i);
\draw[>=stealth,->,black] (U\i) -- (K\i);
}
\foreach \x in {F,U} {
\draw[>=stealth,->,black] (\x1) -- (\x2);
\draw[>=stealth,->,black] (\x2) -- (\x3);
\draw[>=stealth,->,black] (\x3) -- (\x4);
}
\end{tikzpicture}
\]
\end{example}

The fact that the maps between the standard Galois orders exist in the previous example can be described in general by the following:

\begin{corollary}\label{cor: std Galois order morphism}
Given the setting of Theorem \ref{thm: morphisms sufficient condition} and additionally assume
that $\hat{W_2}=\psi(\hat{W_1})\times\hat{H}$ with
$(\psi(w),h)\in\hat{W_2}$ acting on $\varphi(a)\otimes u\in\Lambda_2=\varphi(\Lambda_1)\otimes U$ by

\begin{equation}\label{eq: direct product action on tensor}
(\psi(w),h)(\varphi(a)\otimes u)=\varphi(w(a))\otimes h(u), 
\end{equation}

the map $\Phi\colon\mathcal{F}_{\Lambda_1}\rightarrow\mathcal{F}_{\Lambda_2}$ restricts to their centralizer subalgebras
$\Phi\colon\mathcal{K}_{\Gamma_1}\rightarrow\mathcal{K}_{\Gamma_2}$
\end{corollary}
\begin{proof}
By assumption, $\hat{W_2}=\psi(\hat{W_1})\times\hat{H}$. As such, $\#W_2 = \#W_1\cdot\#H$
and
\[
e_2=\frac{1}{\#W_2}\sum_{h\in\hat{H}}\sum_{w\in W_1}(\psi(w),h)=\frac{\#W_1}{\#W_2}\sum_{h\in\hat{H}}(\psi(e_1),h)
=\frac{1}{\#H}\sum_{h\in\hat{H}}(\psi(e_1),h)
\]
where 
\[
(\psi(e_1),h)=\frac{1}{\#W_1}\sum_{w\in W_1}(\psi(w),h)\quad\text{for }h\in\hat{H}.
\]
By Theorem \ref{thm: morphisms sufficient condition} and Lemma 2.3 from \cite{Webster19},
we have
\[
\mathcal{K}_{\Gamma_2}\cong e_2\mathcal{F}_{\Lambda_2}e_2\supseteq e_2\mathcal{F}_{\Lambda_1}e_2.
\]
We claim that $e_2\mathcal{F}_{\Lambda_1}e_2\cong\mathcal{K}_{\Gamma_1}$. This is clear by the observation
that $e_2=\frac{1}{\#\hat{H}}(\psi(e_1),h)$ made at the beginning of this proof and the required action
of $(\psi(e_1),h)$ from (\ref{eq: direct product action on tensor}).
\end{proof}

\subsection{Split short exact sequences}

We show that certain short exact sequences:
\[
0\rightarrow I\rightarrow\Lambda_2\rightarrow\Lambda_1\rightarrow0,
\]
give rise to embeddings of standard flag orders.

\begin{theorem}\label{thm: standard order intersection}
Let $(\Lambda_1,W_1,\mathscr{M}_1)$ and $(\Lambda_2,W_2,\mathscr{M}_2)$ be
flag order data and $\mathcal{F}_{\Lambda_1},\mathcal{F}^\prime_{\Lambda_2}$ be the
corresponding standard flag orders such that the following are true:
\begin{itemize}
\item $\Lambda_2=\Lambda_1\oplus I$, where $I=(a_1,a_2,\ldots,a_n)$ is an ideal of $\Lambda_2$,
\item there are embeddings $W_1\rightarrow W_2$ and
$\mathscr{M}_1\rightarrow\mathscr{M}_2$ inducing an embedding $\hat{W}_1\rightarrow\hat{W}_2$ that
satisfies the Condition \ref{eq:phipsi-condition} with the natural embedding of $\Lambda_1\rightarrow\Lambda_2$,
\item for every $w\in\hat{W}_1$ and $a_i\in I=(a_1,a_2,\ldots,a_n)$, $w(a_i)=a_i$.
\end{itemize}
Then $\mathcal{F}_{\Lambda_2}\cap\mathcal{F}_1=\mathcal{F}_{\Lambda_1}$. In particular, $\mathcal{F}_{\Lambda_1}\hookrightarrow\mathcal{F}_{\Lambda_2}$.
\end{theorem}

Note that it follows that $I$ is finitely generated due to $\Lambda_2$ being Noetherian.

\begin{proof}
The first two assumptions allow for an embedding
\[
\mathcal{F}_1=\Frac(\Lambda_1)\#\hat{W}_1\rightarrow\Frac(\Lambda_2)\#\hat{W}_2=\mathcal{F}_2.
\]
Thus this intersection is reasonable to consider.

\noindent$\subset\colon$ Let $X\in\mathcal{F}_{\Lambda_2}\cap\mathcal{F}_1$.
First, $X(\Lambda_1)\subset\Lambda_2$ as $\Lambda_1\subset\Lambda_2$ and $X\in\mathcal{F}_{\Lambda_2}$. Second, $X(\Lambda_1)\subset\Frac(\Lambda_1)$.
Hence, 
\[
X(\Lambda_1)\subset\Lambda_2\cap\Frac(\Lambda_1)=(\Lambda_1\oplus I)\cap\Frac(\Lambda_1)=\Lambda_1\oplus(I\cap\Frac(\Lambda_1)).
\]
We claim that $\Frac(\Lambda_1)\cap I=0$. This follows as $I\cap\Lambda_1=0$
and if $I\cap(\Frac(\Lambda_1)\setminus\Lambda_1)\neq0$, then $1\in I$
which is a contradiction. Thus $X\in\mathcal{F}_{\Lambda_1}$.

\noindent$\supset\colon$ Let $X\in\mathcal{F}_{\Lambda_1}$. It is obvious
that $X\in\mathcal{F}_1\subset\mathcal{F}_2$. We need to show that
$X(\Lambda_2)\subset\Lambda_2$. Recall that $\Lambda_2=\Lambda_1\oplus I$ and
$X(a+b)=X(a)+X(b)$.
We observe that $\Lambda_2=\Lambda_1[a_1,a_2,\ldots,a_n]$, where the $a_i$ are the generators of $I$. This can be seen via the following diagram where the outer maps are clearly isomorphisms.
\[
\begin{tikzpicture}
\node(FO1)                      {$0$};
\node(I1)   [right=0.5cm of FO1]  {$(a_1,a_2,\ldots,a_n)$};
\node(C1)   [right=1.25cm of I1]   {$\Lambda_2=\Lambda_1\oplus I$};
\node(L1)   [right=1.5cm of C1]   {$\Lambda_1$};
\node(LO1)  [right=0.5cm of L1]   {$0$};

\node(FO2)  [below=1.1cm of FO1]  {$0$};
\node(I2)   [below=1cm of I1]   {$(a_1,a_2,\ldots,a_n)$};
\node(C2)   [below=1cm of C1]   {$\Lambda_1[a_1,a_2,\ldots,a_n]$};
\node(L2)   [below=1cm of L1]   {$\Lambda_1$};
\node(LO2)  [below=1.1cm of LO1]  {$0$};

\foreach \x in {1,2} {
\draw[->,>=stealth,black] (FO\x) -- (I\x);
\draw[->,>=stealth,black] (I\x) -- (C\x);
\draw[->,>=stealth,black] (C\x) -- (L\x);
\draw[->,>=stealth,black] (L\x) -- (LO\x);
}
\foreach \x in {I,C,L} {
\draw[<->,>=stealth,black] (\x1) -- (\x2);
}

\end{tikzpicture} 
\]
This means that every $f\in\Lambda_2$ can be written as
\begin{equation}\label{eq: poly of lambda2}
f=\sum_{0\leq j_1\leq j_2,\leq\cdots,\leq j_n\leq k} f_ja_1^{j_1}a_2^{j_2}\cdots a_n^{j_n}\quad\text{where }f_j\in\Lambda_1
\end{equation}
By the third assumption, for any $w\in\hat{W}_1$
$a_i\in I$, $w(a_i)=a_i$. This means that
\begin{align*}
X(f)&=X\left(\sum_{0\leq j_1\leq j_2,\leq\cdots,\leq j_n\leq k} f_ja_1^{j_1}a_2^{j_2}\cdots a_n^{j_n}\right)\\
&=\sum_{0\leq j_1\leq j_2,\leq\cdots,\leq j_n\leq k} X(f_ja_1^{j_1}a_2^{j_2}\cdots a_n^{j_n})\\
&=\sum_{0\leq j_1\leq j_2,\leq\cdots,\leq j_n\leq k} a_1^{j_1}a_2^{j_2}\cdots a_n^{j_n}X(f_j)\\
\intertext{Now, $X(f_j)=g_j\in\Lambda_1$ due to $X\in\mathcal{F}_{\Lambda_1}$. Therefore,}\\
&=\sum_{0\leq j_1\leq j_2,\leq\cdots,\leq j_n\leq k} a_1^{j_1}a_2^{j_2}\cdots a_n^{j_n}g_j\\
&\in\Lambda_2.
\end{align*}
Hence,
$X\in\mathcal{F}_{\Lambda_2}\cap\mathcal{F}_1$.
\end{proof}

We apply the above to prove a result inspired by differential operators on affine varieties.

\begin{definition}
Given an ideal $I\subset\Lambda$, we define:
\[
\mathcal{F}_\Lambda[I]=\{X\in\mathcal{F}_\Lambda\mid X(I)\subset I\},
\]
the \emph{subring of $\mathcal{F}_\Lambda$ that fixes $I$}.
\end{definition}

\begin{definition}
Given and ideal $I\subset\Lambda$, we define
\[
I\mathcal{F}_\Lambda=\{X\in\mathcal{F}_\Lambda\mid X(\Lambda)\subset I\},
\]
the \emph{subring of $\mathcal{F}_\Lambda$ send $\Lambda$ to $I$. In fact, $I\mathcal{F}_\Lambda$ is an ideal of $\mathcal{F}_\Lambda[I]$.}
\end{definition}

To see that $I\mathcal{F}_\Lambda$ is an ideal, let $X\in\mathcal{F}_\Lambda[I]$ and
$Y\in I\mathcal{F}_\Lambda$. Then for some $a\in I$,
\[
XY(a)=X(Y(a))\in X(I)\subset I,
\]
so $XY\in\mathcal{F}_\Lambda[I]$. Similarly, $YX\in\mathcal{F}_\Lambda[I]$ by
\[
YX(a)=Y(X(a))\in Y(I)\subset I.
\]

\begin{remark}
We observe that $I\subset I\mathcal{F}_\Lambda$ based on the fact that the action of $\Lambda$ on itself is multiplication. 
\end{remark}

\begin{remark}
We warn the reader that $I\mathcal{F}_\Lambda\supsetneq I\cdot\mathcal{F}_\Lambda$ in general.
\end{remark}

\begin{lemma}\label{lem: quotient injects into End}
The map $\mathcal{F}_{\Lambda_2}[I]/I\mathcal{F}_{\Lambda_2}\rightarrow\End(\Lambda_1)$
is injective.
\end{lemma}
\begin{proof}
First we observe that $\mathcal{F}_{\Lambda_2}[I]\rightarrow\End(\Lambda_1)$ by
sending $X\mapsto(a+I\mapsto X(a)+I)$. We now claim the kernel of this map is
$K=I\mathcal{F}_{\Lambda_2}$. It is clear that $K\supset I\mathcal{F}_{\Lambda_2}$, and if
$X\in K$ then $X(a+I)=I$, that is $X(a)\in I$ for all $a\in\Lambda_1$. Since
$\Lambda_2=\Lambda_1\oplus I$, it follows that $X\in I\mathcal{F}_{\Lambda_2}$.
Hence the map is injective.
\end{proof}

\begin{theorem}\label{thm: standard order quotient}
Following the same assumptions as in Theorem \ref{thm: standard order intersection}, we have an embedding 
$\eta\colon\mathcal{F}_{\Lambda_1}
\hookrightarrow\mathcal{F}_{\Lambda_2}[I]/I\mathcal{F}_{\Lambda_2}$
\end{theorem}
\begin{proof}
In the proof of Theorem \ref{thm: standard order intersection} it was shown that
$\mathcal{F}_{\Lambda_1}\hookrightarrow\mathcal{F}_{\Lambda_2}[I]$, and it is known that
$F_{\Lambda_1}\hookrightarrow\End(\Lambda_1)$. This gives rise to the following diagram:
\[
\begin{tikzpicture}
\node(E)                        {$\End(\Lambda_1)$};
\node(T)    [above=0.75cm of E]  {$\mathcal{F}_{\Lambda_2}[I]$};
\node(Q)    [left=1cm of E]   {$\mathcal{F}_{\Lambda_2}[I]/I\mathcal{F}_{\Lambda_2}$};
\node(S)    [right=1cm of E]  {$\mathcal{F}_{\Lambda_1}$};

\draw[>=stealth, right hook->,black] (Q) -- (E);
\draw[>=stealth, <-left hook,black] (E) -- (S);
\draw[>=stealth, ->,black] (T) -- (E);
\draw[>=stealth, <-left hook,black] (T)--(S);
\draw[>=stealth, ->>,black] (T) -- (Q);
\end{tikzpicture}
\]
The left triangle arises from Lemma \ref{lem: quotient injects into End} and clearly commutes. Now the right triangle commutes because for all $a\in\Lambda_1$, $X(a)=X(a+I)$ by definition. Thus the whole triangle commutes, and
$\mathcal{F}_{\Lambda_1}
\hookrightarrow\mathcal{F}_{\Lambda_2}[I]/I\mathcal{F}_{\Lambda_2}$.
\end{proof}

\begin{example}
We can apply this result to the same set-up as Example \ref{ex: Klein 4->S4} in this setting as $\C[x_1,x_2,x_3,x_4]=\C[x_2-x_1,x_4-x_3]\oplus(x_2+x_1,x_4+x_3)$, our choice of maps in Example \ref{ex: Klein 4->S4} satisfy the other two requirements. Therefore, Theorem \ref{thm: standard order quotient} applies. We describe the components of the RHS for the benefit of the reader:
\[
\mathcal{F}_{\C[x_1,x_2,x_3,x_4]}[(x_2+x_1,x_4+x_3)]
=\C[x_1,x_2,x_3,x_4]\left\langle\frac{1}{x_2-x_1}((12)-\mathbbm{1}),\frac{1}{x_4-x_3}((34)-\mathbbm{1})\right\rangle
\]
and
\[
(x_2+x_1,x_4+x_3)\mathcal{F}_{\C[x_1,x_2,x_3,x_4]}=(x_2+x_1,x_4+x_3)\cdot\mathcal{F}_{\C[x_1,x_2,x_3,x_4]}.
\]
\end{example}

While the map $\eta$ in Theorem \ref{thm: standard order quotient} is surjective in our previous example, this is not generally true.
This is unlike the situation of differential operators on polynomial rings. Even if $\Lambda_2$ is a polynomial ring and $\hat{W}_2$
a complex reflection group. The following example demonstrates this.

\begin{example}\label{example: counter example to quotient isomorphism}
Let $\Lambda_2=\C[x_1,x_2,x_3]$, $\Lambda_1=\C[x_1]$, $I=(x_2,x_3)$, $\hat{W}_2=S_3$
acting by permutation of variables, and $\hat{W}_1$ trivial. In this case
$\mathcal{F}_{\Lambda_1}\subsetneq\mathcal{F}_{\Lambda_2}[I]/I\mathcal{F}_{\Lambda_2}$,
as the permutation $(23)$ is on the right-hand side, but is not in the image of $\eta$ as
$\hat{W}_1$ is trivial.
\end{example}


\section{Tensor Products}\label{sec: tensors}

Let $(\Lambda_i,\mathscr{M}_i,W_i)$ for $i=1,2$ be the data for standard flag orders
$\mathcal{F}_{\Lambda_i}\subset\mathcal{F}_i=\Frac(\Lambda_i)\#\hat{W}_i$, where
$\hat{W}_i=W_i\ltimes\mathscr{M}_i$. Let $\Lambda=\Lambda_1\otimes\Lambda_2$,
$\mathscr{M}=\mathscr{M}_1\times\mathscr{M}_2$,
$W=W_1\times W_2$, and $\mathcal{F}=\Frac(\Lambda)\#\hat{W}$, where
$\hat{W}=W\ltimes\mathscr{M}=\hat{W}_1\times\hat{W}_2$.

The following is a generalization of Lemma 2.17 (ii) from \cite{HARTWIG2020106806}.

\begin{lemma}\label{lem: nonzero determinant}
Given a collection of elements $\{X_i\}_{i=1}^n\in\mathcal{F}$ that
are linearly independent over $\Frac(\Lambda)$, then there exists $\{a_i\}_{i=1}^n\in\Lambda$ such that
\[
\det\bigg(\big(X_i(a_j)\big)_{i,j=1}^n\bigg)\neq0
\]
\end{lemma}
\begin{proof}
Identical to the proof of Lemma 2.17 (ii) in \cite{HARTWIG2020106806} by using $\sigma_i=X_i$, $A=\Lambda$, and $F=\Frac(\Lambda)$.
\end{proof}

\begin{lemma}\label{lem: aj are simple tensors}
When applying Lemma 2.17 (ii) from \cite{HARTWIG2020106806} to $A=\Lambda$ and
$F=\Frac(\Lambda)$, and $\sigma_1,\ldots,\sigma_n\in W\ltimes\mathscr{M}$ the choices of
$(a_1,a_2,\ldots,a_n)\in\Lambda^n$ can be selected such that $a_j$ is a simple tensor
for each $j=1,2,\ldots,n$.
\end{lemma}
\begin{proof}
We use induction on $n$. For $n=1$, since $\sigma_1\in\hat{W}$ acts as an automorphism
of $\Lambda$, it is nonzero on the simple tensor $1\otimes1$. For $n>1$, we assume
we have simple tensors $(a_1,a_2,\ldots,a_{n-1})\in\Lambda^{n-1}$ such that $(\sigma_j(a_i))_{i,j=1}^{n-1}$ has nonzero
determinant. We now observe by part (i) of Lemma 2.17 from \cite{HARTWIG2020106806} that there exists an $a_n\in\Lambda$
such that
\[
(\sigma_n-\sum_{i=1}^{n-1}x_i\sigma_i)(a_n)\neq0.
\]
We claim that we can choose $a_n$ to be a simple tensor. If for the sake of argument we
assume that $\sigma_n-\sum_{i=1}^{n-1}x_i\sigma_i$ is zero on every simple tensor,
then if $a_n=\sum_{j=1}^k a_j^{(1)}\otimes a_j^{(2)}$ is a sum of simple tensors, where
$a_j^{(i)}\in\Lambda_i$,
\[
0\neq(\sigma_n-\sum_{i=1}^{n-1}x_i\sigma_i)(a_n)=\sum_{j=1}^k(\sigma_n-\sum_{i=1}^{n-1}x_i\sigma_i)(a_j^{(1)}\otimes a_j^{(2)})=0.
\]
Which is a contradiction.
\end{proof}

\begin{notation}
Below, if $A$ is an algebra action on a vector space $V$, and $W\subset V$ is a subspace, then we put
$A_W=\{a\in A\mid aW\subset W\}$.
\end{notation}

Recall that the standard Galois order $\mathscr{K}_\Gamma$ can be regarded as a spherical subalgebra of $\mathcal{F}_\Lambda$, as
$\mathscr{K}_\Gamma\cong e\mathcal{F}_\Lambda e$, where
$e=\frac{1}{\#W}\sum_{w\in W}w$ \cite{Webster19}.

\begin{theorem}\label{thm: tensor of std orders}
\text{}
\begin{enumerate}[\rm (a)]
\item There is a chain of embeddings
\[
\mathcal{F}_{\Lambda_1}\otimes\mathcal{F}_{\Lambda_2}
\hookrightarrow\mathcal{F}_\Lambda\hookrightarrow
(\mathcal{F}_1\otimes\mathcal{F}_2)_\Lambda.
\]
\item There is a chain of embeddings
\[
\mathscr{K}_{\Gamma_1}\otimes\mathscr{K}_{\Gamma_2}
\hookrightarrow\mathscr{K}_\Gamma\hookrightarrow
(\mathscr{K}_1\otimes\mathscr{K}_2)_\Gamma.
\]
\end{enumerate}
\end{theorem}
\begin{proof}
(a) First we observe the following is an embedding of algebras:
\[
\psi\colon\mathcal{F}_1\otimes\mathcal{F}_2\hookrightarrow\mathcal{F}
\]
by $X_1(w_1,\mu_1)\otimes X_2(w_2,\mu_2)\mapsto X_1((w_1,\mu_1),(1,1))X_2((1,1),(w_2,\mu_2))$ and extending linearly.
If we restrict this embedding to $\mathcal{F}_{\Lambda_1}\otimes\mathcal{F}_{\Lambda_2}$,
this gives us an embedding
\[
\widetilde{\psi}\colon\mathcal{F}_{\Lambda_1}\otimes\mathcal{F}_{\Lambda_2}\hookrightarrow\mathcal{F}_\Lambda.
\]
To see this, we just need to show that $\psi(X_1\otimes X_2)\big(\Lambda\big)\subset\Lambda$ for
$X_1\otimes X_2\in\mathcal{F}_{\Lambda_1}\otimes\mathcal{F}_{\Lambda_2}$. However, this
holds since
$X_1\otimes X_2(\lambda_1\otimes\lambda_2)=X_1(\lambda_1)\otimes X_2(\lambda_2)\in\Lambda_1\otimes\Lambda_2$ for all
$\lambda_1\in\Lambda_1,\lambda_2\in\Lambda_2$.

Next, we show the second embedding. We first observe what happens when applying Lemma 2.17 from \cite{HARTWIG2020106806}
to a $X\in \mathcal{F}_\Lambda$. We observe that $X=\sum_{i=1}^kf_i(w_{i1},w_{i2})$ where
$f_i\in\Frac(\Lambda)$ and $(w_{i1},w_{i2})\in\hat{W}$. Let $n=\vert\{w\in\hat{W}_1\mid\exists w^\prime\in\hat{W}_2\colon (w,w^\prime)\in\supp_{\hat{W}}{X}\}\vert$ and
$m=\vert\{w^\prime\in\hat{W}_2\mid\exists w\in\hat{W}_1\colon(w,w^\prime)\in\supp_{\hat{W}}{X}\}\vert$. WLOG we can assume
that $k=n\cdot m$. Let $\{a_{j1}\}$ (resp. $\{a_{j2}\})$ be the set of elements of
$\Lambda_1$ (resp. $\Lambda_2$) such that the matrix $A_1:=(w_{i1}(a_{j1}))_{i,j=1}^n$
(resp. $A_2:=(w_{i2}(a_{j2}))_{i,j=1}^m$) has non-zero determinant denoted $d_1$
(resp. $d_2$). Then the matrix $A=\big((w_{i1},w_{i2})(a_{j1}\otimes a_{j2})\big)$ has
non-zero determinant; moreover, it is clear that $A=A_1\otimes A_2$ so the determinant
is $d=d_1^m\otimes d_2^n$. As such, if $A^\prime$ is the adjugate of $A$ ($A^\prime\cdot A=d\cdot I_{k}$), then it follows using $A^\prime$ that for each $i=1,\ldots,k$:
\[
f_i\in\frac{1}{d}\Lambda=\frac{1}{d_1^m}\Lambda_1\otimes\frac{1}{d_2^n}\Lambda_2.
\]
This shows us that $X\in\psi(\mathcal{F}_1\otimes\mathcal{F}_2)$; moreover, $X\in \psi\big((\mathcal{F}_1\otimes\mathcal{F}_2)_{\Lambda}\big)$. This leads to the second embedding.\\

(b) The symmetrizing idempotent in the group algebra of $W$ can be factored as $e=e_1e_2=e_2e_1$, where $e_i=\frac{1}{\#W_i}\sum_{w\in W_i} w$ for $i=1,2$. Thus $e\mapsto e_1\otimes e_2$. Therefore, by part (a) and using Webster's observation \cite{Webster19} that $e\mathcal{F}_\Lambda e\cong \mathcal{K}_\Gamma$ and
$e_i\mathcal{F}_{\Lambda_i} e_i\cong \mathcal{K}_{\Gamma_i}$ for $i=1,2$, 
this proves the claim.
\end{proof}

\begin{remark}
In all examples we know of, the map $\widetilde{\psi}
\colon\mathcal{F}_{\Lambda_1}\otimes\mathcal{F}_{\Lambda_2}\hookrightarrow\mathcal{F}_\Lambda$ is surjective, making
$\widetilde{\psi}$ an isomorphism.
\end{remark}

\begin{example}
Let $\Lambda=\C[x_1,x_2,\ldots,x_n]$, $W$ trivial, and $\mathscr{M}\cong\Z^n$, then $\mathcal{F}_\Lambda=\Lambda\#\Z^n\cong A_n(\C)$ the $n$-th Weyl algebra. As is well-known $A_n(\C)\otimes A_m(\C)\cong A_{n+m}(\C)$.
\end{example}

\begin{example}
Let $\mathscr{M}$ be trivial. Then $\hat{W}=W$ is finite and $L\#W\cong\End_{\Lambda^W}(L)=\End_{\Lambda^W}(L)$ \cite{HersteinBook}, hence $(L\# W)_\Lambda=\End_{\Lambda^W}(\Lambda)=\End_\Gamma(\Lambda)$.
As such, if $\mathscr{M}_1,\mathscr{M}_2$ trivial, then
\begin{align*}
\mathcal{F}_{\Lambda_1}\otimes\mathcal{F}_{\Lambda_2}
\cong\End_{\Gamma_1}(\Lambda_1)\otimes\End_{\Gamma_2}(\Lambda_2)&\cong\End_{\Gamma_1\otimes\Gamma_2}(\Lambda_1\otimes\Lambda_2)\cong\mathcal{F}_\Lambda,\\
\intertext{via}
\Phi\vert_{\Lambda_1\otimes1}\otimes\Phi\vert_{1\otimes\Lambda_2}&\mapsfrom\Phi\\
\Psi_1\otimes\Psi_2&\mapsto\big((a_1\otimes a_2)\mapsto\Psi_1(a_1)\otimes\Psi_2(a_2)\big)
\end{align*}
\end{example}

Theorem \ref{thm: tensor of std orders} gives us the two very useful corollaries.

\begin{corollary}\label{cor: principal flag orders closed under tensor}
Principal flag orders are closed under tensor products. That is,
given two principal flag orders $F_1$ and $F_2$ with data $(\Lambda_1,W_1,\mathcal{M}_1)$ and
$(\Lambda_2,W_2,\mathcal{M}_2)$ respectively, with each $\Lambda_i$ a $\K$-algebra for a ground field $\K$,
$F_1\otimes_{\K}F_2$ is a principal flag order with data
$(\Lambda_1\otimes\Lambda_2,W_1\times W_2,\mathcal{M}_1\times\mathcal{M}_2)$.
\end{corollary}
\begin{proof}
We need to show that $F_1\otimes F_2$ satisfies the 3 conditions from Definition
\ref{def: principal flag order}. It is clear that $F_i\subset\mathcal{F}_{\Lambda_i}$
for each $i$ from the definition of the standard flag order. As such,
$F_1\otimes F_2\subset \mathcal{F}_{\Lambda_1\otimes\Lambda_2}$ via the embedding from
Theorem \ref{thm: tensor of std orders} which proves the third condition. Also,
$\Lambda_i\# W_i\subset F_i$ for each $i$, so we satisfy the first condition that
$(\Lambda_1\otimes\Lambda_2)\#(W_1\times W_2)\cong(\Lambda_1\#W_1)\otimes(\Lambda_2\#W_2)\subset F_1\otimes F_2$.

All that remains to prove is the second condition. To do this we observe the following:
\[
\Frac(\Lambda_1\otimes\Lambda_2)=\Frac(\Frac(\Lambda_1)\otimes\Frac(\Lambda_2))
\]
which implies that
\[(\Lambda_1\otimes\Lambda_2)^{-1}(\Frac(\Lambda_1)\otimes\Frac(\Lambda_2))
=\Frac(\Lambda_1\otimes\Lambda_2)
\]
Therefore, we see that
\begin{align}
(\Lambda_1\otimes\Lambda_2)^{-1}(F_1\otimes F_2)&=(\Lambda_1\otimes\Lambda_2)^{-1}((\Frac(\Lambda_1)\#(W_1\ltimes\mathscr{M}_1))\otimes(\Frac(\Lambda_2)\#(W_2\ltimes\mathscr{M}_2)))\label{eq: localize from principal condition}\\
&=(\Lambda_1\otimes\Lambda_2)^{-1}((\Frac(\Lambda_1)\otimes\Frac(\Lambda_2))\#((W_1\times W_2)\ltimes(\mathscr{M}_1\times\mathscr{M}_2))\nonumber\\
&=\Frac(\Lambda_1\otimes\Lambda_2)\#((W_1\times W_2)\ltimes(\mathscr{M}_1\times\mathscr{M}_2))\label{eq: field of fracs the same}
\end{align}
with \ref{eq: localize from principal condition} follows from each $F_i$ being a principal flag order and \ref{eq: field of fracs the same} follows from the above fact. Thus proving the second condition and that $F_1\otimes F_2$ is a principal flag order.
\end{proof}

\begin{corollary}\label{cor: principal galois orders closed under tensor}
Principal Galois orders are closed under tensor products. More explicitly, given two
principal Galois $\Gamma_1$-order $\mathscr{U}_1$ and a principal Galois $\Gamma_2$-order $\mathscr{U}_2$, with each $\Lambda_i$
a $\K$-algebra for a ground field $\K$, $\mathscr{U}_1\otimes_{\K}\mathscr{U}_2$ is a Galois $\Gamma_1\otimes\Gamma_2$-order.
\end{corollary}
\begin{proof}
Essentially the same as the proof of the previous corollary with the observation that
\[
(\Frac(\Lambda_1)\#\mathscr{M}_1)^{W_1}\otimes(\Frac(\Lambda_2)\#\mathscr{M}_2)^{W_2}
=((\Frac(\Lambda_1)\otimes\Frac(\Lambda_2))\#(\mathscr{M}_1\times\mathscr{M}_2))^{W_1\times W_2}
\]
based on the action of $W_1\times W_2$.
\end{proof}


\section*{Acknowledgements}
The author would like to thank his advisor Jonas Hartwig for some comments and helpful
discussion. Finally, the author would like to thank Iowa State University, where the author resided during some of the results in this paper.

\printbibliography

\end{document}